\documentclass[a4paper,12pt,oneside]{article}
\usepackage{graphicx}

\usepackage[centertags]{amsmath}
\usepackage{graphicx}
\usepackage{amsfonts}
\usepackage{amssymb}
\usepackage[all]{xy}
\usepackage{color}
\usepackage{enumerate}

\usepackage[utf8]{inputenc}
\usepackage{tikz-cd}
\usepackage{hyperref}
\usepackage{mathtools}
\usepackage{import}

\usepackage{relsize}
\usepackage[all]{xy}
\usepackage{microtype}
\usepackage{xspace}
\usepackage{soul}
\usepackage{xcolor}

\usepackage{geometry}

\def\squarebox#1{\hbox to #1{\hfill\vbox to #1{\vfill}}}
\newcommand{\qed}{\hspace*{\fill}
\vbox{\hrule\hbox{\vrule\squarebox{.667em}\vrule}\hrule}\smallskip}

\newtheorem{theorem}{Theorem}[section]
\newtheorem{lemma}[theorem]{Lemma}
\newtheorem{corollary}[theorem]{Corollary}
\newtheorem{proposition}[theorem]{Proposition}

\newtheorem{example}[theorem]{Example}

\newenvironment{proof}{\noindent {\bf Proof:}}{\hfill $\qed $ \newline}

\newcommand{\R}{{\mathbb R}}
\newcommand{\N}{{\mathbb N}}
\newcommand{\Z}{{\mathbb Z}}

\newcommand{\F}{\mathbb{F}}

\renewcommand{\t}{{\mathbb T}}
\newcommand{\simto}{\stackrel{\sim}{\longrightarrow}}

\newcommand{\ad}{{\rm ad}}
\newcommand{\Ad}{{\rm Ad}}

\newcommand{\Sl}{{\rm Sl}}
\newcommand{\SO}{{\rm SO}}

\newcommand{\cl}{\mathrm{cl}}
\newcommand{\e}{{\rm e}}

\renewcommand{\min}{\mbox{{\rm min}}}

\newcommand{\g}{\mathfrak{g}}
\renewcommand{\k}{\mathfrak{k}}
\newcommand{\s}{\mathfrak{s}}
\renewcommand{\a}{\mathfrak{a}}

\newcommand{\m}{\mathfrak{m}}
\renewcommand{\v}{\mathfrak{v}}
\renewcommand{\r}{\mathfrak{r}}

\newcommand{\n}{\mathfrak{n}}
\renewcommand{\l}{\mathfrak{l}}

\newcommand{\q}{\mathfrak{q}}

\newcommand{\uniaodisj}{{\coprod}}

\def\prodi#1{\left\langle{#1}\right\rangle}

\begin{document}
\thispagestyle{empty}

\title{Dynamics of translations on \\ maximal compact subgroups}

\author{Mauro Patrão \and Ricardo Sandoval}

\maketitle

\begin{abstract}
In this article, we study the dynamics of translations of an element of a semisimple Lie group $G$ acting on its maximal compact subgroup $K$. First, we extend to our context some classical results in the context of general flag manifolds, showing that when the element is hyperbolic its dynamics is gradient and its fixed points components are given by some suitable right cosets of the centralizer of the element in $K$. Second, we consider the dynamics of a general element and characterizes its recurrent set, its minimal Morse components and their stable and unstable manifolds in terms of the Jordan decomposition of the element, and we show that each minimal Morse component is normally hyperbolic.
\end{abstract}


\section{Introduction}

\setcounter{page}{1}

The dynamics of translations induced by an element of a connected semisimple Lie group acting on some of its compact homogeneous spaces is extensively studied in the literature. For example, the iteration of a linear isomorphism acting on the projective space is a classical topic, with applications to the study of skew product flows on projective bundles (Selgrade Theorem, see \cite{sel}), even though some interesting results, such as the normal hyperbolicity of the minimal Morse components on the projective space, were established only recently, using techniques from Lie groups. More generally, the dynamics of translations of an element of a linear connected semisimple Lie group acting on a generalized flag manifold is reasonably understood, and there are applications of these results to semigroup theory which allows us to obtain the full generalization of Selgrade Theorem to skew product flows on generalized flag bundles (see \cite{psm}). This broader context of generalized flag manifolds encompasses other interesting cases such as the classical flag manifolds of real or complex nested subspaces and also symplectic grassmanians, which were extensively studied in the literature (see for example \cite{ammar, ayala, batterson, dkv, ferraiol, kleinsteuber, hermann, psm, pss, sel, shub}). Despite this, there are many interesting dynamics of translations on compact homogeneous spaces of a connected semisimple Lie group which are not generalized flag manifolds, such as spheres or grassmanians of oriented subspaces, in the case of the special linear group. But all these examples of dynamics are projections of dynamics of translations on the respective maximal compact subgroups.

In this article, we study the dynamics of translations $g^t$ acting on a maximal compact subgroup $K$ of a connected semisimple Lie group $G$ with finite center. More precisely, when $\t = \R$, we have that $g^t = \exp(tX)$ and, when $\t = \Z$, we have that $g^t$ is the $t$-iterate of $g$. When $t=1$ we just write $g=g^1$. We look at $K$ as the homogeneous space $G/AN$, where $G = KAN$ is the assotiated Iwasawa decomposition. If $b = AN$ is the base point of $G/AN$, the the map $k \mapsto kb$ is a diffeomorphism from $K$ to $G/AN$. The action of $g^t$ on $K$ is given through the natural action of $G$ on $G/AN$.  Equivalently, this action is given by $(g^t,k) \mapsto \kappa(g^tk)$, where $\kappa: G \to K$ is the corresponding Iwasawa projection. In the case of the special linear group, this projection corresponds to applying the Gram-Schmidt process to the columns of the matrix, since the Iwasawa decomposition corresponds in this case to the QR decomposition.

The first main set of results of the present paper extends to the maximal compact subgroups the results proved in \cite{dkv}, concerning the dynamics of translations $h^t = \exp(tH)$, where $H \in \g$ is a hyperbolic element,  acting on the flag manifolds of $G$, showing that this is a gradient dynamics, and describing its fixed points components and respective stable manifolds. Observe that, differently from the flag manifold situation, in the maximal compact subgroup case, all the fixed points components are diffeomorphic, in fact, they are right cosets of a compact and connected subgroup.

\begin{theorem}\label{theorem1}
The dynamics of the hyperbolic translation $h^t$ is gradient on $K$.
\begin{enumerate}[$(i)$]

\item The connected components of the fixed points of $h^t$ are given by
\[
 {\rm fix} (H, u) = K_H^0 ub
\]
where $K_H^0$ is the connected component of the identity of the centralizer of $H$ in $K$ and $u$ is an element of the group $U$ of connected components of the normalizer of $A$ in $K$.

\item The set of fixed points is the disjoint union of connected components
\[
{\rm fix}(h^t) = \uniaodisj \left\{ {\rm fix}(H, u) : \, u \in U_H \backslash U \right\}
\]
where $U_H$ is the centrilizer of $H$ in $U$. The attractors are given by ${\rm fix} (H, c)$, where $c$ is an element of the subgroup $C$ of $U$, given by the connected components of the centralizer of $A$ in $K$.

\item The stable manifold of ${\rm fix} (H, u)$ is given by
\[
 {\rm st} (H, u) = N^-_HK_H^0 ub
\]
where $N^-_H$ is the subgroup generated by roots which have negative value at $H$.

\end{enumerate}

\end{theorem}

The second main set of results of the present article extends to the maximal compact subgroups the results proved in \cite{ferraiol,Patrao1}, concerning the dynamics of general translations $g^t$ acting on the flag manifolds of $G$, describing its recurrent set, its minimal Morse components and respective stable manifolds through the Jordan decomposition of $g^t$, and showing that theses components are normally hyperbolic.

\begin{theorem}\label{theorem2}
Let $g^t$ be a general translation on $K \simeq G/AN$ and $g^t = e^t h^t u^t$ be its Jordan decomposition. Then
\begin{enumerate}[$(i)$]
\item The recurrent set is given by
\[
\mathcal{R}(g^t) = {\rm fix} (h^t) \cap {\rm fix} (u^t)
\]

\item The minimal Morse components of $g^t$ are normally hyperbolic and given by
\[
 {\cal M}(g^t,u) = {\rm fix}(H, u) \cup g{\rm fix}(H, u)
\]
If $g$ is in the connected component of the identity $G_H^0$ of the centralizer of $H$ in $G$, which always happens when $\mathbb{T} = \R$, then ${\cal M}(g^t,u)$ is connected and equal to ${\rm fix}(H, u)$. If $g \notin G_H^0$, then ${\cal M}(g^t,u)$ has two connected components ${\rm fix}(H, u)$ and $g{\rm fix}(H, u)$.

\item The chain recurrent set is given by
\[
{\cal R}_C (g^t) = {\rm fix} (h^t) = \uniaodisj \left\{ {\cal M}(g^t,u): \, u \in U_H^g \backslash U \right\}
\]
where
\[
U_H^g := U_H \cup c_gU_H
\]
for a element $c_g \in C$ such that ${\rm fix}(H, c_gu) = g{\rm fix}(H, u)$. The attractors are given by ${\cal M}(g^t,c)$ where $c \in C$.

\item The stable manifold of ${\cal M}(g^t,u)$ is given by
\[
 {\rm st}(g^t,u) := N^-_H {\cal M}(g^t,u)
\]
\end{enumerate}
\end{theorem}

The structure of the article is the following. In Section 2, we set the notation and recall the necessary definitions and results about dynamics, homogeneous spaces of Lie groups, semi-simple Lie theory and dynamics of translations on maximal flag manifolds. In Section 3, we prove Theorem \ref{theorem1} and, in Section 4, we prove Theorem \ref{theorem2}. We end this introduction presenting some low dimensionsional examples where the results can be easily visualized and a higher dimensionsional example with a richer dynamics.

\begin{example}
Let $G = \Sl(2,\R)$ and maximal compact subgroup
\[
K
=
\SO(2)
=
\left\{
k
=
\begin{pmatrix}
\cos \alpha & -\sin \alpha \\
\sin \alpha & \cos \alpha
\end{pmatrix}
: \alpha \in \R \right\}
\]
with the other Iwasawa subgroups components given by:
\[
A
=
\left\{
a
=
\begin{pmatrix}
x & 0 \\
0 & x^{-1}
\end{pmatrix}
: x>0 \right\}
\qquad
\mbox{and}
\qquad
N = \left\{
n
=
\begin{pmatrix}
1 & x \\
0 & 1
\end{pmatrix}
: x \in \R \right\}
\]
First we observe that we can identify $K = \SO(2)$ with the unity circle $S^1$ in $\r^2$, since each element of $K = \SO(2)$ is determined by its first column, which is a point in $S^1$. We have that
\[
U
=
\left\{
\begin{pmatrix}
1 & 0 \\
0 & 1
\end{pmatrix},
\begin{pmatrix}
-1 & 0 \\
0 & -1
\end{pmatrix},
\begin{pmatrix}
0 & -1 \\
1 & 0
\end{pmatrix},
\begin{pmatrix}
0 & 1 \\
-1 & 0
\end{pmatrix}
\right\}
\]
while
\[
C
=
\left\{
\begin{pmatrix}
1 & 0 \\
0 & 1
\end{pmatrix},
\begin{pmatrix}
-1 & 0 \\
0 & -1
\end{pmatrix}
\right\}
\]

Now we consider the hyperbolic action on the compact group $\SO(2)$ given by:
\[
h^t
=
\exp(tH)
=
\begin{pmatrix}
\e^t & 0 \\
0 & \e^{-t}
\end{pmatrix},
\qquad
\mbox{where}
\qquad
H
=
\begin{pmatrix}
1 & 0 \\
0 & -1
\end{pmatrix}
\]
whose action is given by:
\[
\kappa(h^tk)
=
\begin{pmatrix}
\cos\alpha/\sqrt{\cos^2\alpha+\e^{-4t}\sin^2\alpha} & * \\
\sin\alpha/\sqrt{\e^{4t}\cos^2\alpha+\sin^2\alpha} & *
\end{pmatrix}
\]
where we applied the Gram-Schmidt process in the first column of the following the matrix:
\[
h^tk
=
\begin{pmatrix}
\e^t\cos\alpha & -\e^t\sin \alpha \\
\e^{-t}\sin\alpha & \e^{-t}\cos \alpha
\end{pmatrix}
\]
Now it is easy to see that this dynamics is gradient, where the attractors are the points with $\sin \alpha=0$ and the repellers are the points with $\cos \alpha =0$. The first column of the matrices is used to plot the results in matrix form.

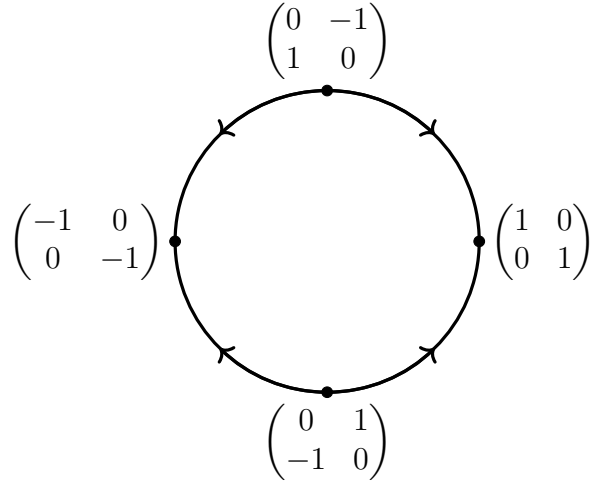
\begin{center}
\begin{tikzpicture}

\filldraw[color=black!100, fill=red!0, very thick](6,0) circle (2);
\draw[very thick, ->] (6,-2) arc (-90:-45:2);
\draw[very thick, ->] (6,-2) arc (270:225:2);
\draw[very thick, ->] (6,2) arc (90:135:2);
\draw[very thick, ->] (6,2) arc (90:45:2);

\filldraw [black] (8,0) circle (2pt) node[anchor=west] {$\begin{pmatrix}
1 & 0 \\
0 & 1
\end{pmatrix}$};

\filldraw [black] (4,0) circle (2pt) node[anchor=east] {$\begin{pmatrix}
-1 & 0 \\
0 & -1
\end{pmatrix}$};
\filldraw [black] (6,2) circle (2pt) node[anchor=south] {$\begin{pmatrix}
0 & -1 \\
1 & 0
\end{pmatrix}$};
\filldraw [black] (6,-2) circle (2pt) node[anchor=north] {$\begin{pmatrix}
0 & 1 \\
-1 & 0
\end{pmatrix}$};
\end{tikzpicture}
\end{center}
Since $K_H^0$ is trivial, the fixed points components are isolated points given by the elements of $U$, while the two attractors are given by the elements of $C$.

If $\t = \Z$, we can consider the dynamics given by the translations of the following element
\[
g^t
=
\begin{pmatrix}
(-\e)^t & 0 \\
0 & (-\e)^{-t}
\end{pmatrix}
\]
whose action is given by:
\[
\kappa(g^tk)
=
\begin{pmatrix}
(-1)^t\cos\alpha/\sqrt{\cos^2\alpha+\e^{-4t}\sin^2\alpha} & * \\
(-1)^t\sin\alpha/\sqrt{\e^{4t}\cos^2\alpha+\sin^2\alpha} & *
\end{pmatrix}
\]
where we applied the Gram-Schmidt process in the first column of the following the matrix:
\[
g^tk
=
\begin{pmatrix}
(-\e)^t\cos\alpha & -(-\e)^t\sin \alpha \\
(-\e)^{-t}\sin\alpha & (-\e)^{-t}\cos \alpha
\end{pmatrix}
\]
Since $g^t = h^te^t$ is the Jordan decomposition, where the elliptic component is given by
\[
e^t
=
\begin{pmatrix}
(-1)^t & 0 \\
0 & (-1)^{-t}
\end{pmatrix}
\]
the dynamics is not gradient anymore, there are only two minimal Morse components, where the attractor is given by $C = \{I,c_g\}$, with
\[
c_g
=
\begin{pmatrix}
-1 & 0 \\
0 & -1
\end{pmatrix}
\]
and the repeller is given by
\[
\left\{
\begin{pmatrix}
0 & -1 \\
1 & 0
\end{pmatrix},
\begin{pmatrix}
0 & 1 \\
-1 & 0
\end{pmatrix}
\right\}
\]
both with two connected components (two isolated points).

Now we consider the unipotent action on the compact group $\SO(2)$ given by:
\[
u^t
=
\exp(tX)
=
\begin{pmatrix}
1 & t \\
0 & 1
\end{pmatrix},
\qquad
\mbox{where}
\qquad
X
=
\begin{pmatrix}
0 & 1 \\
0 & 0
\end{pmatrix}
\]
whose action is given by:
\[
\kappa(u^tk)
=
\begin{pmatrix}
((\cos\alpha)/t + \sin\alpha)/\sqrt{1/t^2+(\sin2\alpha)/t+\sin^2\alpha} & * \\
\sin\alpha/\sqrt{1+t\sin2\alpha+t^2\sin^2\alpha} & *
\end{pmatrix}
\]
where we applied the Gram-Schmidt process in the first column of the following the matrix:
\[
u^tk
=
\begin{pmatrix}
\cos\alpha + t\sin\alpha &- \sin \alpha +t\cos \alpha  \\
\sin\alpha & \cos\alpha
\end{pmatrix}
\]
The system is not gradient, but it is chain transitive, while the recurrent set coincides with the fixed points, which are given by $C$.

\begin{center}
\begin{tikzpicture}

\filldraw[color=black!100, fill=red!0, very thick](6,0) circle (2);

\draw[very thick, ->] (8,0) arc (0:-90:2);
\draw[very thick, ->] (4,0) arc (180:90:2);

\filldraw [black] (8,0) circle (2pt) node[anchor=west] {$\begin{pmatrix}
1 & 0 \\
0 & 1
\end{pmatrix}$};
\filldraw [black] (4,0) circle (2pt) node[anchor=east] {$\begin{pmatrix}
-1 & 0 \\
0 & -1
\end{pmatrix}$};
\end{tikzpicture}
\end{center}

\end{example}

\begin{example}
Let $G = \Sl(3,\R)$ and maximal compact subgroup $K = \SO(3)$, with the other Iwasawa components given by the subgroup $A$ of diagonal matrices with positive diagonal elements and the subgroup $N$ of upper triangular matrices with diagonal elements equal to one. We have that $U$ has 24 elements which are given by permutation matrices with signal and determinant one such as
\[
\begin{pmatrix}
 0 & 1 &  0 \\
 0 & 0 & -1 \\
-1 & 0 &  0
\end{pmatrix}
\]
while
\[
C
=
\left\{
\begin{pmatrix}
1 & 0 & 0 \\
0 & 1 & 0 \\
0 & 0 & 1
\end{pmatrix},
\begin{pmatrix}
-1 & 0 & 0 \\
0 & -1 & 0 \\
0 & 0 & 1
\end{pmatrix},
\begin{pmatrix}
1 & 0 & 0 \\
0 & -1 & 0 \\
0 & 0 & -1
\end{pmatrix},
\begin{pmatrix}
-1 & 0 & 0 \\
0 & 1 & 0 \\
0 & 0 & -1
\end{pmatrix}
\right\}
\]

Now consider the dynamics given by the translations on $\SO(3)$ of the following element
\[
g^t
=
\begin{pmatrix}
\e^{2t} & 0 & 0 \\
0 & \e^{-t} & t\e^{-t} \\
0 & 0 & \e^{-t}
\end{pmatrix}
\]
Its Jordan decomposition is given by $g^t = h^tu^t$, where
\[
h^t
=
\exp(tH)
=
\begin{pmatrix}
\e^{2t} & 0 & 0 \\
0 & \e^{-t} & 0 \\
0 & 0 & \e^{-t}
\end{pmatrix},
\qquad
\mbox{where}
\qquad
H
=
\begin{pmatrix}
2 & 0 & 0 \\
0 & -1 & 0 \\
0 & 0 & -1
\end{pmatrix}
\]
and
\[
u^t
=
\begin{pmatrix}
1 & 0 & 0 \\
0 & 1 & t \\
0 & 0 & 1
\end{pmatrix}
\]
Since
\[
K_H^0
=
\left\{
\begin{pmatrix}
1 & 0 & 0 \\
0 & \cos \alpha & -\sin \alpha \\
0 & \sin \alpha & \cos \alpha
\end{pmatrix}
: \alpha \in \R
\right\}
\]
which is isomorphic to $\SO(2)$, and
\[
U_H
=
\left\{
\begin{pmatrix}
1 & 0 & 0 \\
0 & 1 & 0 \\
0 & 0 & 1
\end{pmatrix},
\begin{pmatrix}
1 & 0 & 0 \\
0 & -1 & 0 \\
0 & 0 & -1
\end{pmatrix},
\begin{pmatrix}
1 & 0 & 0 \\
0 & 0 & 1 \\
0 & -1 & 0
\end{pmatrix},
\begin{pmatrix}
1 & 0 & 0 \\
0 & 0 & -1 \\
0 & 1 & 0
\end{pmatrix}
\right\}
\]
there are 6 minimal Morse components given by $K_H^0u$, for $u \in U_H \backslash U$, all diffeomorphic to the circle. The recurrent set is given by ${\rm fix} (h^t) \cap {\rm fix} (u^t)$, which has 12 elements, two in each minimal Morse component. For example, inside $K_H^0$, the recurrent set is given by
\[
\left\{
\begin{pmatrix}
1 & 0 & 0 \\
0 & 1 & 0 \\
0 & 0 & 1
\end{pmatrix},
\begin{pmatrix}
1 & 0 & 0 \\
0 & -1 & 0 \\
0 & 0 & -1
\end{pmatrix}
\right\}
\]

\end{example}

\section{Preliminaries}\label{secpreliminar}

\subsection{Dynamics}

We recall some concepts of topological dynamics (for more
details, see \cite{conley,pugh}). Let $\phi :\mathbb{T}\times F\to F$
be a continuous dynamical system on a compact metric space $(F, d)$, with
discrete $\mathbb{T}={\mathbb Z}$ or continuous
$\mathbb{T}={\mathbb R}$ time.
Denote by $\omega(x)$, $\omega^*(x)$, respectively, the forward and
backward omega limit sets of $x$. A Morse decomposition of $\phi^t$, which is given by a finite collection of disjoint subsets $
\{{\mathcal M}_{1},\ldots ,{\mathcal M}_{n}\}$ of $F$ such that
\begin{enumerate}[$(i)$]
\item each ${\mathcal M}_i$ is compact and
$\phi\,^t$-invariant,

\item for all $x \in F$ we have $\omega(x),\, \omega^*(x)
\subset \bigcup_i {\mathcal M}_i$,

\item if $\omega(x),\, \omega^*(x) \subset {\mathcal M}_j$
then $x \in {\mathcal M}_j$.
\end{enumerate}
The minimal Morse decomposition is a Morse decomposition which is contained in every other Morse decomposition. Each set ${\mathcal M}_i$ of a minimal Morse decomposition is called a minimal Morse component. The stable/unstable set of a morse component ${\mathcal M}_i$ is the set of all points whose forward/backward omega limit set is contained in ${\mathcal M}_i$. An $(\epsilon, t)$-{\em chain} from $x$ to $y$ is a sequence of points
\[
\{ x = x_0, \dots, x_n = y \} \subset F
\]
and a sequence of times $t_i$ such that $t_i \geq t$ and $d(\phi^{t_i} (x_i), x_{i+1}) < \epsilon$. The {\em chain recurrent set} of a flow $\phi^t$, ${\cal R_C }(\phi^t)$, is the set of points $x$ such that there is $(\epsilon, t)$-chain for every $\epsilon >0$ and $t>0$ from $x$ to $x$. A {\em minimal} Morse decomposition is a decomposition that is contained in any other Morse decomposition. Each element of the decomposition is also called a {\em component} of the decomposition. An important result from dynamical systems (see \cite{conley}) is that if the Morse components are connected and
\[
\cup_i {\cal M}_i = {\cal R_C} (\phi^t)
\]
then this decomposition is minimal. The {\em recurrent set} of a flow $\phi^t$ in a space $F$ is the set of points
\[
{\cal R} (\phi^t ) := \{ x \in X : x \in \omega (x) \}
\]

Now assume that $\phi$ is a diffeomorphism on a Riemannian manifold $F$ and $D\phi$ its derivative. An invariant submanifold ${\cal M} \subset F$ is normally hyperbolic if the tangent bundle of $F$ over ${\cal M}$ has invariant vector subbundles $V^+$ and $V^-$ and positive constants $c$ and $\nu < \lambda$ such that
\begin{enumerate}[$(i)$]
\item $TF|_{\cal M} =  T{\cal M} \oplus V^- \oplus V^+ $
\item $|D \phi^n v| \leq c e^{-\lambda n} |v| \text{ for all } v \in V^- \text{ and } n \geq 0$
\item $|D \phi^n v| \leq c e^{\lambda n} |v| \text{ for all } v \in V^+ \text{ and } n \leq 0$
\item $|D \phi^n v| \leq c e^{\nu |n|} |v| \text{ for all } v \in T{\cal M} \text{ and } n \in \Z$
\end{enumerate}
in this case, $V^-$ is said to be the stable bundle and $V^+$ the unstable bundle of ${\cal M}$.
If $\phi^t$ is a differentiable flow on $F$, $t \in \R$, an invariant submanifold ${\cal M}$ is normally hyperbolic if its is normally hyperbolic for the time one diffeomorphism $\phi^1$. If the flow $\phi^t$ is gradient relative to some height function $f : F \to \R$, the following sets coincide:
\begin{enumerate}[$(i)$]
\item The critical points of $f$.

\item The fixed points of $\phi^t$, $t \in \t$.
\end{enumerate}

\subsection{Homogeneous spaces of Lie groups}
\label{homogspaces}

For the theory of Lie groups and its homogeneous spaces we refer to Hilgert and Neeb \cite{neeb} and Knapp \cite{knapp} and for the theory of principal bundles we refer to Steenrod \cite{steenrod}.  Let $G$ be a real Lie group with Lie algebra $\g$ where $g \in G$ acts on $X \in \g$ by the adjoint action
$g X = \Ad( g ) X$.  We have that  $\Ad( \exp(X) ) = e^{\ad(X)}$
where $\exp: \g \to G$ is the exponential of $G$, $\Ad$ and $\ad$ are, respectively, the adjoint representation of $G$ and $\g$.

Let a Lie group $G$ act on a manifold $F$ on the left by the differentiable map
$
G \times F \to F$,  $(g,x) \mapsto gx
$.
Fix a point $x \in F$.  The isotropy subgroup $G_{x}$ is the set of all $g \in G$ such that $gx = x$.  We say that the action is transitive or, equivalently, that $F$ is a homogeneous space of $G$, if $F$ equals the orbit $Gx$ of $x$ (and hence the orbit of every point of $F$).
In this case, the map
$$
G \to F\qquad g \mapsto g x
$$
is a submersion onto $F$ which is a differentiable locally trivial principal fiber bundle with structure group the isotropy subgroup $G_{x}$.
Quotienting by $G_{x}$ we get the diffeomorphism
$$
G/G_{x} \simto F\qquad g G_{x} \mapsto g x
$$

If $L$ is a Lie subgroup of $G$, the orbit $L x$ is the set of all $hx$, $h \in L$. The restriction of the principal fiber bundle $G \to F$ to $L$ gives the submersion onto the orbit $L x$
$$
L \to L x \qquad l \mapsto l x
$$
which is a differentiable locally trivial principal fiber bundle with structure group $L_x = L \cap G_{x}$.  If $L x$ is an embedded submanifold of $F$ then around every point in $L x$ there exists a differentiable local section from $L x$ to $L$ that is a restriction of a local section from $F$ to $G$ of the principal fiber bundle $G \to F$.

Since the map $G \to F$ is a submersion, the derivative of the map $g \mapsto gx$ on the identity gives the infinitesimal action of $\g$, more precisely, a surjective linear map
$$
\g \to TF_x \qquad Y \mapsto Y \cdot x
$$
whose kernel is the isotropy subalgebra $\g_x$, the Lie algebra of $G_x$.
The derivative of the map $g: F \to F$, $x \mapsto gx$, gives the action of
$G$ on tangent vectors $g v = D g (v)$, $v \in TF$, which is related to the infinitesimal action by
$$
g( Y \cdot x ) = gY \cdot gx
$$
For a subset $ \q \subset \g$, denote by $\q \cdot x$ the set of all tangent vectors $Y \cdot x$, $Y \in \q$.
It follows that $ TF_{gx} = g( \g \cdot x ) $. In particular, for $l \in L$, the tangent space of the orbit $L x$ at $l x$ is given by $l (\l \cdot x) \subset TF_{lx}$, where $\l \subset \g$ is the Lie algebra of $L$. Thus, the tangent bundle of the orbit is given by
$$
T(L x) = L( \l \cdot x )
$$
Let $E$ be another manifold with a differentiable action of $G$, a map $f: F \to E$ is said to be $G$-equivariant if $f(g x) = g f(x)$. Such a $G$-equivariant map is automatically differentiable.

\subsection{Semi-simple Lie theory}\label{section-lie}

For the theory of real semisimple Lie groups and their flag manifolds we refer to Duistermat-Kolk-Varadarajan \cite{dkv}, Hilgert and Neeb \cite{neeb}, and Knapp \cite{knapp}.
Let $G$ be a connected real Lie group with semi-simple Lie algebra $\g$ and finite center.
Fix a Cartan decomposition $\g = \k \oplus \s$ and denote by $\theta$ the Cartan involution and by $\langle \cdot, \cdot \rangle$ the associated Cartan inner product.
Let $K$ be the connected subgroup with Lie algebra $\k$, it is a maximal compact subgroup of $G$.
Since $\ad(X)$ is anti-symmetric for $X \in \k$, the Cartan inner product is $K$-invariant.
Since $\ad(X)$ is symmetric for $X \in \s$, a maximal abelian subspace $\frak{a} \subset \frak{s}$ can be simultaneously diagonalized so that $\g$ splits as an orthogonal sum of
$$
\g_\alpha = \{ X \in \g:\, \ad(H)X = \alpha(H)X, \, \forall H \in \a \}
$$
where $\alpha \in \a^*$ (the dual of $\a$). We have that $\g_0 = \m \oplus \a$, where $\m$ is the centralizer of $\a$ in $\k$. A root is a functional $\alpha \neq 0$ such that its root space $\g_\alpha \neq 0$, denote the set of roots by $\Pi$. We thus have the root space decomposition of $\g$, given by the orthogonal sum
$$
\g = \m \oplus \a \oplus \sum_{\alpha \in \Pi} \g_\alpha
$$
Fix a Weyl chamber $\frak{a}^{+}\subset \frak{a}$ and let $\Pi^{+}$ be the corresponding positive roots, $\Pi^- = - \Pi^+$ the negative roots and $\Sigma $ the set of simple roots.  Consider the nilpotent subalgebras
\[
\n^\pm = \sum_{\alpha \in \Pi^{\pm}}\frak{g}_{\alpha }
\]
such that
\[
\g = \m \oplus \a \oplus \n \oplus \n^-
\]

In this paper, we look at $K$ as the homogeneous manifold $G/AN$, denoting its base, the left coset $AN$, as $b$. The natural action of $G$ in $G/AN$ is given by left multiplication (as in section 10.1 of \cite{neeb}). From the Iwasawa decomposition, it follows that the map $K \to G/AN$ given by $k \mapsto k b$, is a $K$-equivariant diffeomorphism. The isotropy subalgebra of the base $b$ is given by $\a \oplus \n$, while the isotropy subalgebra of $k b$ for $k \in K$ is given by
\[
\g_{k b} = k (\a \oplus \n)
\]
We also look at the maximal flag manifold $\F$ of $G$ as the homogeneous manifold $G/MAN$, denoting its base, the left coset $MAN$, as $b_0$. and is diffeomorphic to $K/M$. The natural projection $K \to \F$ given by $kb \mapsto k b_0$, is a $M$-principal bundle with the action of the structural group given $(kb,m) \mapsto kmb$.

The Weyl group $W$ is the finite group generated by the reflections over the root hyperplanes $\alpha=0$ in $\frak{a}$, $\alpha \in \Pi$. $W$ acts on $\frak{a}$ by isometries and can be alternatively be given as $W=M^{*}/M$ where $M^{*}$ and $M$ are the normalizer and the centralizer of $\a$ in $K$,
respectively.
An element $w$ of the Weyl group $W$ can act in $\g$ by taking a representative in $M^*$.  This action centralizes $\a$, normalizes $\m$, permutes the roots $\Pi$ and thus permutes the root spaces $\g_\alpha$, where $w \g_\alpha = \g_{w \alpha}$ does not depend on the representative chosen in $M^*$.

Let $H \in \frak{a}$ and denote the centralizer of $H$ in $G$ and $K$ respectively by $G_{H}$ and $K_H$. We have that
\[
G_{H} = K_{H} A (G_{H} \cap N)
\]
and that
\[
 G_{H} = G_{H}^0M
 \qquad \mbox{and} \qquad
 K_{H} = K_{H}^0M
\]
where $G_{H}^0$ and $K_{H}^0$ are the connected components of the identity respectively of $G_{H}$ and $K_H$.
Consider the nilpotent subalgebras
\[
\frak{n}^\pm_H = \sum_{\pm \alpha(H) > 0} \g_\alpha
\]
given by the the sum of the positive/negative eigenspaces of $\ad(H)$ in $\g$ and let $N_H^{\pm}$ be the corresponding connected Lie subgroups. Since $G_H$ leaves invariant each eigenspace of $\ad(H)$ it follows that $\frak{n}^\pm_H$ and $N_H^{\pm}$ are $G_H$-invariant. If $Z \in \s$ centralizes $H$, then there exists $k \in K_H^0$ such that $kZ \in \a$ and, if $Z \in \a$ and $kZ \in \a$ for some $k \in K$, then there exists $m \in M_*$ such that $kZ = mZ$.

\subsection{Translations on maximal flag manifolds}\label{subsec-translations}

Here we collect some previous results about the dynamics of a flow $g^t$ of translations of a real semisimple Lie group $G$ acting on its maximal flag manifold $\F$. The flow $g^t$ is either given by the iteration of some $g \in G$, for $t \in \Z$, or by $\exp(tX)$, for $t \in \R$, where $X \in \g$, and $g^t$ acts on $\F$ by left translations. Since $G$ acts on its flag manifolds by the adjoint action we will assume that $G$ is a linear Lie group, and thus $\g$ is linear Lie algebra.

For the description of the flow $h^t = \exp (tH)$, $t \in \R$, induced by $H\in \mathrm{cl}\frak{a}^{+}$ on the flag manifold ${\mathbb F}_{\Theta }$ see (\cite{dkv}, Section 3). The connected components of the fixed points of $h^t$ are labeled by $w \in W$, each one given by the orbit
\[
\mathrm{fix}_\F(H,w) = G_{H} wb_0 = K_{H} wb_0,
\]
which is an compact submanifold of $\F$, and its unstable/stable manifold of are given by
\[
N_H^{\pm} \mathrm{fix}_\F(H,w)
\]

The usual additive Jordan decomposition writes a matrix as a commuting sum of a semisimple and a nilpotent matrix and we can decompose the semisimple part further as the commuting sum of its imaginary and its real part, where each part commutes with the nilpotent part and the matrix is diagonalizable over the complex numbers iff its nilpotent part is zero.
This generalizes to a multiplicative Jordan decomposition of
the flow $g^t$ in the semisimplie Lie group $G$ (see Section 2.3 of \cite{ferraiol}), providing us with a commutative decomposition
\[
g^t = e^t h^t u^t
\]
where there exist a Cartan decomposition of $\g$ with a corresponding maximal compact subgroup $K$ and a Weyl chamber $\a^+$ such that the elliptic component $e^t$ lies in $K$, the hyperbolic component is such that $h^t = \exp(tH)$, where $H \in \cl \a^+$, and the unipotent component is such that $u^t = \exp(tN)$, with $N \in \g$ nilpotent. Furthermore, we have that $h^t$, $e^t$ and $u^t$ lie in $G_H$, the centralizer of $H$ in $G$, and that $g^t$ is diagonalizable iff its unipotent part is $u^t = 1$. We have that, for all $kb_0 \in F$ there is $k'_0 \in K$ so that $u^t k b_0 \to k'_0 b_0$ when $ t \to \pm \infty$.
We have that the hyperbolic component $H$ dictates the minimal Morse components (see Proposition 5.1 and Theorem 5.2 of \cite{ferraiol}).

\begin{proposition}
The minimal Morse components of $g^t$ on $\F$ are given by ${\rm fix}_\F(H,w)$, $w \in W$, and their unstable/stable manifolds are given by $N^\pm_H {\rm fix}_\F(H,w)$.
\end{proposition}

\section{Hyperbolic translations}

In this section, we describe the dynamics of the hyperbolic component $h^t = \exp (t H)$, $t \in \R$ and $H \in \cl \a^+$, defined in the subsection \ref{subsec-translations}.

\subsection{Borel metric}

Let $H_r \in \a^+$ be a \emph{regular} element. We have that $M$ is the isotropy at $H_r$ by the adjoint action of $K$ on $\s$. Now, let $H$ be a \emph{fixed} element of $\cl \a^+$ and define {\em the height function} $f_H$ as:
\[
f_H: K = G/AN \to \R,\qquad kb \mapsto \langle{kH_r, H\rangle}
\]

\begin{proposition}\label{propfuncaoaltura}
The function $f_{H}$ is $K_H$-invariant and its differential is
\[
f_{H}'(k b)k(Z \cdot b) = \langle{[Z, H_r],\, k^{-1} H \rangle}
\]
where $k \in K$ and $Z \in \k$.
\end{proposition}
\begin{proof}
If $l \in K_H$, since where $l$ acts as a $\langle \cdot, \cdot \rangle$-orthogonal map, we have that
\[
f_H (lkb) = \langle{ lkH_r, H\rangle } = \langle{ kH_r, l^{-1} H\rangle } = \langle{ kH_r, H\rangle} = f_H (kb)
\]
Let us evaluate its differential at $k b$ in the direction $k Z$ with $Z \in \k$,
\begin{eqnarray*}
f_{H}'(k b) k (Z \cdot b)
& = & \left. d/dt \right|_{t=0}\,f_{H}(k \exp(t Z) b) \\
& = & \left. d/dt \right|_{t=0}\,\langle{k \exp(t Z) H_r,\, H \rangle} \\
& = & \left. d/dt \right|_{t=0}\,\langle{ \e^{t\ad(Z)} H_r,\,k^{-1} H \rangle} \\
& = & \langle{ [Z, H_r],\, k^{-1} H \rangle}
\end{eqnarray*}
since $k$ is $\langle \cdot, \cdot \rangle$-orthogonal.
\end{proof}

Now, our next objective is to define a metric in $K$ such that the field induced by $H$ in $K$ is the the gradient of the height function $f_H$. Since the isotropy subalgebra at $b$ is given by $\a \oplus \n$ and we have the orthogonal decomposition
\[ 
\g = \m \oplus \a \oplus \n \oplus \n^-
\]
we have that
\begin{equation}
\label{eqperpisotropia1}
(\a \oplus \n)^\perp = \m \oplus \n^-
\end{equation}
can be viewed as the the tangent space of $K = G/AN$ at $b$. Similar to the proof of Proposition 3.1 of \cite{Patrao1}, if $B(X, Y)$ is an inner profuct on $\m \oplus \n^-$,  we have that
\[
B_{k b}(k (X \cdot b), k (Y \cdot b)) := B(X, Y) \mbox{ where } X,Y \in \m \oplus \n^- 
\]
defines a $K$-invariant Riemannian metric of $K = G/AN$. Let $\lambda$ be a real number, and ${\frak b}_\lambda$ be the $\lambda$-eigenspace of $\ad(H_r)$ in $\g$ given by
\[
{\frak b}_\lambda := \{ X \in \g: \, \ad (H_r) X = \lambda X \}
\]
Note that 
\[
{\frak b}_\lambda = \sum_{\alpha (H_r) = \lambda}\g_\alpha
\]
and
\[
\n^- = \sum_{\lambda < 0} {\frak b}_\lambda
\]
For $X \in \g$ let $X_\lambda$ be the orthogonal projection of $X$ in ${\frak b}_\lambda$ and let $X_0$ be the orthogonal projection of $X$ in $\m$. Let $c_\lambda$ and $c_0$ be positive real numbers associated to ${\frak b}_\lambda$ and $\m$, respectively. Lets define the inner product in $\m \oplus\n^-$ given by
\begin{equation} \label{metric B}
B (X,Y) := \sum_{\lambda \leq 0} c_\lambda \langle X_\lambda, Y_\lambda \rangle \mbox{ where } X, Y \in \m \oplus \n^-
\end{equation}

\begin{theorem}\label{teo:gradiente}
Taking $c_0$ arbitrary positive and $c_\lambda = - 2 \lambda$ for $\lambda < 0$ in equation \textnormal{\ref{metric B}}, then
\[
H \cdot= \nabla_{B} f_H
\]
that is, the induced field by $H \in \s$ in $K$ is the gradient of the height function $f_H$ with respect to the $K$-invariant metric $B$. Also,
\[
B (X,Y) = c_0 \langle X_0,Y_0 \rangle + 2 \langle [X, H_r], Y \rangle \qquad X, Y \in \m \oplus \n^-
\]
where $X_0$ and $Y_0$ are the components of $X$ and $Y$ in $\m$.
\end{theorem}

\begin{proof}
By the definition of $\nabla_B$ for $k \in K$ and $X \in \m \oplus \n^-$,
\[
B_{k b}(k (X \cdot b), \nabla_B f_H (k b)) = f_{H}'(k b) k (X \cdot b)
\]
So to prove the first statement we need to show that
\begin{equation}\label{eq0:gradaltura}
B_{k b}\left( k (X \cdot b),\, H \cdot k b \right) = f_{H}'(k b) k (X \cdot b)
\end{equation}

To evaluate the left side, let $Y_-$ be the orthogonal projection of $k^{-1}H$ in $\m \oplus \n^-$, that is parallel to $\a \oplus \n$. Note that since $k^{-1}H \in \s$ then $Y_- \in \n^-$.

Then $Y_- \cdot b = k ^{-1} H \cdot b$, and $k(Y_- \cdot b) = H \cdot k b$. Let $X = X_0 + X_-$ where $X_0 \in \m$ and $X_- \in \n^-$. By the $K$-invariance of the metric the left side is
\begin{align} \nonumber
B_{k b}\left( k ( X \cdot b),\, H \cdot kb \right) &= B_{k b}\left( k ( X \cdot b),\, k(Y_- \cdot b) \right) \\ \nonumber
&= B (X_-,Y_-) \\ 
&= \sum_{\lambda < 0} c_\lambda \langle X_\lambda ,Y_\lambda \rangle
\label{eq1:metricaborel}
\end{align}

To evaluate the right side, let $Z = X_0 + X_- + \theta X_- \in \k$. Since $\theta X_- \in \n$, then $Z \cdot b = (X_0 + X_-) \cdot b$. By Proposition \ref{propfuncaoaltura} then
\begin{equation}\label{eq2:metricaborel}
f_{H}'(k b) k (( X_0 + X_-) \cdot b)  = \langle{[Z,H_r],\, k^{-1}H \rangle}
\end{equation}
To evaluate $[Z, H_r]$, first note that
\[
[H_r, X_-] = \sum_{\lambda < 0} \ad(H_r)X_\lambda = \sum_{\lambda < 0} \lambda X_\lambda
\]
\[
[H_r, \theta X_-] = - [\theta H_r, \theta X_-] = - \theta [H_r, X_-] = - \sum_{\lambda < 0} \lambda \theta X_\lambda
\]
and that $[H_r, X_0]=0$. Then
\[
[Z, H_r] = \sum_{\lambda < 0} \lambda (\theta X_\lambda - X_\lambda)
\]
Since $k^{-1}H \in \s$, then
\[
\langle{ \theta X_\lambda ,\, k^{-1}H \rangle} = - \langle{ \theta X_\lambda ,\, \theta k^{-1}H \rangle} = - \langle{ X_\lambda ,\,
k^{-1}H \rangle}
\]
and
\[
\langle{ [Z, H_r],\, k^{-1}H \rangle} = \sum_{\lambda < 0} -2\lambda \langle{ X_\lambda ,\, k^{-1} H \rangle} = \sum_{\lambda < 0} -2\lambda \langle{ X_\lambda ,\, Y_\lambda \rangle}
\]
since $Y_- = \sum_{\lambda<0} Y_{\lambda} $ is the projection of $k^{-1}H$ at $\m \oplus \n^-$. From (\ref{eq2:metricaborel}) and since $ X = X_0 + X_- $ then
\begin{equation}\label{eq3:metricaborel}
f_{H}'(kb) k \left( X \cdot b \right) = \sum_{\lambda < 0} -2\lambda \langle{ X_\lambda ,\, Y_\lambda \rangle}
\end{equation}
Now, equation (\ref{eq0:gradaltura}) follows by (\ref{eq1:metricaborel}) and (\ref{eq3:metricaborel}), since $c_{\lambda}= -2\lambda$ for $\lambda <0$.

In order to prove the last statement, let $X_-,Y_- \in \n^-$ then
\[
\begin{array}{rcl}
B (X_-, Y_-) & = & -2 \sum_{\lambda < 0} \lambda \langle{ X_\lambda ,\, Y_\lambda \rangle} \\
& = & -2 \langle{ \sum_{\lambda < 0} \lambda X_\lambda ,\, Y_- \rangle} \\
& = & -2 \langle{ [H_r, X] ,\, Y_- \rangle} \\
& = & 2 \langle{ [X, H_r] ,\, Y_- \rangle}
\end{array}
\]
Since $Y = Y_0 + Y_-$ with $Y_0 \in \m$ then 
\[
B (X, Y) = B (X_0,Y_0) + B (X_-,Y_-)
\]
\end{proof}

The Riemannian metric $B$ constructed on the last theorem is an extension of the \emph{Borel metric} of $\F$.

\subsection{Fixed points}

In this subsection, we will describe the fixed points of $h^t$ in $K$ as orbits of $G_H^0$, \emph{the identity component} of the centralizer of $H$ in $G$. In this description, the orbit of $b$ by $M_* = N(K, \a)$ has a central role.
In order to study the fixed points in $K$ it is convenient to define the \emph{group}
\[
U := M_* / M_0
\]
This group will play a similar role to the Weyl group in the study of fixed points in flags. If we consider the subgrup
\[
C:= M / M_0
\]
we have that the Weyl group $W = M_*/M = U / C$ and we denote by $\pi$ the natural projection from $U$ to $W$. Since $M_0$ is a normal subgroup of both $M_*$ and $M$, the elements of the quotients can be view as left or right classes of $M_0$. We will denote an element of the quotients by any of its representatives. For example, $u \in U$ denotes $uM_0 = M_0u$ with $u \in M_*$. The definitions of $U$ and $C$ have some useful consequences. For $u \in U$, we have that $u A = A u$, that $u G_{H} u^{-1} = G_{wH}$, and that $u K_{H} u^{-1} = K_{wH}$, if $\pi(u) = w$, for all $H \in \a$. For $c \in C$, we have that $c H = H$, for all $H \in \a$, that $c \g_\alpha = \g_\alpha$, for any $\alpha \in \Pi$, so $c N c^{-1} = N$.
In analogy to $w^-$, the principal involution, we fix some $u^- \in U$ such that $\pi(u^-) = w^-$, so that $u^- \n = \n^-$ and $u^- N (u^-)^{-1} = N^-$.

We will show that the connected components of the set of fixed points of $h^t$ in $K$ are given by
\[
{\rm fix}(H, u) := G_H^0 u b
\]
where $G_H^0$ is the component of identity of $G_H$ and $u$ varries in $U$, since $M_0 \subset G_H^0$.

\begin{proposition}\label{propptosfixosgH}
If $K^0_H$ is the component of identity of $K_H$, then
\[
G_H^0 u b = K_H^0 u b
\]
is a connected compact submanifold of $K$.
\end{proposition}
\begin{proof}
In order to show the first equality, first we note that
\[
G_{w^{-1}H} = K_{w^{-1}H} A (G_{w^{-1} H} \cap N)
\]
Let $u \in U$ such that $\pi(u) = w$. By taking the conjugate of the equation above by $u$ then
\[
G_H = K_H A (G_H \cap u N u^{-1})
\]
Now, since $A u N u^{-1}u b = u A N b = u b$, it follows that
\[
G_H u b = K_H A (G_H \cap u N u^{-1}) u b \subset K_H u b
\]
Since the action of $G$ in $K$ is continuous then $G_H^0 u b \subset K_H^0 u b$ and since $K_H^0 \subset G_H^0$ then $G_H^0 u b = K_H^0 u b$.
\end{proof}

\begin{theorem}\label{teo:ptos-fixos}
The set of the fixed points of $h^t$ in $K$ are given $K_H^0Ub$.
\end{theorem}
\begin{proof}
The fixed points of $h^t$ in $K$ are the critical points of the height function $f_{H}$.
To get the critical points of $f_{H}$ we rewrite the derivative at the point $k b$, $k \in K$, in the direction $Z \in \k$, given by  Proposition \ref{propfuncaoaltura}, as follows 
\[
\begin{array}{rcl}
f_H' (k b) k (Z \cdot b)
& = & \prodi{ k [Z, H_r],\, H} \\
& = & \prodi{ [k Z, k H_r],\, H} \\
& = & - \prodi{ \ad(k H_r) k Z,\, H} \\
& = & - \prodi{ k Z,\, \ad(k H_r) H} \\
& = & - \prodi{ k Z,\, [k H_r, H] }
\end{array}
\]
where we used that $\langle \cdot, \cdot \rangle$ is $K$-invariant, $k H_r \in \s$ and $\ad(k H_r)$ is $\langle \cdot, \cdot \rangle$-symmetric. Since $k \k = \k$, then $k b$ is a critical point of $f_H$ if, and only if,
\[
\prodi{ Z,\, [k H_r, H]} = 0, \quad\text{ for all }Z \in \k
\]
Since $k H_r$ and $H \in \s$ then $[k H_r, H] \in \k$ and, since $\prodi{\cdot , \cdot}$ is an inner product in $\k$, then the previous equation is true if, and only if, $[k H_r, H] = 0$. Then $k b$ is a critical point of $f_H$ if, and only if, $k H_r$ centralizes $H$.
Since $[kH_r, H] = 0$ for $k \in M_*$, we have that the orbit $M_*b$ consists of critical points of $f_H$. Since $f_H$ is $K_H$-invariant to the left and since $M_0 \subset K_H^0$, then the orbit $K_H^0 Ub$ consists of critical points. In fact, the points of $K_H b$ are all the critical points of $f_H$. Indeed, if $k b$ is critical, $k \in K$, from the previous argument then $k H_r$ centralizes $H \in \cl \a^+$. There is $l \in K_H^0 $ such that $l k H_r \in \a$ and thus there is $m \in M_*$ such that $l k H_r = m H_r$. Then $m^{-1} l k H_r = H_r$ and
\[
m^{-1} l k \in K_{H_r} = M
\]
Hence
\[
k = l^{-1} m ( m^{-1}l k ) \in K_H^0 M_* = K_H^0 U
\]
and thus $K_H^0 U b$ are the only critical points of $f_H$.
\end{proof}

In order to determine the distinct fixed point components, we need to introduce the following subgroup of $U$, given by
\[
U_H: = \frac{K_H^0 \cap M_*}{M_0}
\]

\begin{corollary}\label{corol:fix-flag}
${\rm fix}(H, u) = {\rm fix}(H, v)$ if and only if $v \in U_H u$.
\end{corollary}

\begin{proof}
On the one hand, if $v \in U_H u$, then there are $v = u'u$, where $u' \in U_H$. Thus $u' \in K_H^0 \cap M_*$. Hence
\[
{\rm fix} (H, v) = K_H^0 v b = K_H^0 u'u b = K_H^0 u b = {\rm fix} (H,u)
\]
On the other hand, ${\rm fix} (H, v) = {\rm fix} (H,u)$, then $K_H^0 v b = K_H^0 u b$ and thus $vb = kub$, where $u' \in K_H^0$.Thus $b = v^{-1} u' u b $ so $v^{-1} u' u \in K \cap A N = 1$ and $v = u' u$. Hence $u' = vu^{-1} \in K_H^0 \cap M_*$ and thus $v \in U_H u$.
\end{proof}

\subsection{Bruhat Decomposition}

In this subsection, we show that the stable and unstable sets of each component of fixed points ${\rm fix}(H, u)$ are given respectively by the following orbits
\[
N^-_H {\rm fix}(H, u)
\qquad
\mbox{and}
\qquad
N^+_H {\rm fix}(H, u)
\]
This will provide a decomposition of $K$ which we regard as a general Bruhat decomposition. First we recall the following result about the dynamics of $h^t$ acting on the subalgebras $\n^\pm_H$, which is exactly Lemma 3.3 of \cite{Patrao1}.

\begin{lemma}\label{lemmadecaimentoexp}
Let $h = \exp H$, where $H \neq 0$. Then
\[
|h^t Y| \leq {\rm e}^{-\mu t} |Y| \quad \mbox{for } Y \in \n^-_H, \, t \geq 0
\]
and
\[
|h^{-t} Y| \leq {\rm e}^{- \mu t} |Y| \quad \mbox{for } Y \in \n^+_H, \, t \geq 0
\]
where
\[
\mu = \min\{\alpha(H) :  \alpha(H)>0, \, \alpha \in \Pi\}
\]
\end{lemma}

Consider $N^\pm(H)$ the connected subgroups with the following Lie algebras
\[
\frak{n}^\pm(H) = \sum_{\alpha \in \Pi^\pm: \alpha(H) = 0} \g_\alpha
\]
Note that $N^\pm(H)$ is a subgroup of $G_H^0$, since $\frak{n}^\pm(H)$ is clearly a subalgebra of $\frak{g}_H$. Now we prove the following result relating the subgroups $N^\pm$, $N^\pm_H$ and $N^\pm(H)$.

\begin{lemma}
\label{lemmaestavel}
It follows that:
\begin{enumerate}[(i)]
\item $N^\pm_H = \{ n \in N^-:\, \lim_{t \to \pm \infty} h^t n h^{-t} = 1 \}$.

\item $N^\pm = N^\pm_H N^\pm(H)$.

\item $N^\pm {\rm fix}(H, u) = N^\pm_H {\rm fix}(H, u)$.

\item If $y \in N^\pm_H x$, for $x \in {\rm fix}(H, u)$, then
\[
\lim_{t \to \pm \infty} h^t y = x
\]
\end{enumerate}
\end{lemma}
\begin{proof}
We prove the results only when the superscript is $-$. The proof when the superscript is $+$ is entirely analogous. To prove (i), first we use that $\exp: \n^- \to N^-$ is a homeomorphism, and that $\n^- = \n^-_H \oplus \n^- (H)$, since $H \in \mathrm{cl}(\a^+)$. Then for any $n \in N^-$, $n = \exp(Y_1 + Y_2)$ with $Y_1 \in \n^-_H$ and $Y_2 \in \n^- (H)$, from Lemma \ref{lemmadecaimentoexp} then
\begin{equation}
\label{eqlemmaestavel}
h^t n h^{-t} = \exp(h^t Y_1 + Y_2) \to \exp(Y_2) \in N^- (H)
\end{equation}
when $t \to \infty$, since $h^t$ centralizes $\n^- (H)$. So that $h^t n h^{-t} \to 1$, if and only if, $Y_2 = 0$, that is equivalent to $Y = Y_1 \in \n^-_H$ and $n \in N^-_H$. To prove that $N^- = N^-_H N^- (H)$ let us prove the inclusion ``$\subset$''since the other side is immediate. Let $n \in N^-$ and $n = \exp(Y_1 + Y_2)$ as previously. Consider $n \exp(-Y_2) \in N^-$, since $h^t$ centralizes $Y_2$ then by equation (\ref{eqlemmaestavel})
\[
h^t (n \exp(-Y_2)) h^{-t} = (h^t n h^{-t}) \exp(-Y_2) \to  \exp(Y_2)  \exp(-Y_2) = 1
\]
so that from item (i), $n \exp(-Y_2) = n_1 \in N^-_H$, and $n = n_1 n_2$ where $n_2 = \exp(Y_2) \in N^- (H)$.
Now (iii) is a immediate consequence o (ii), since ${\rm fix}(H, u) = G_H^0ub$ and $N^-(H) \subset G_H^0$.
For (iv), let $y = n x$, with $n \in N^-_H$ and $h^{-t} x = x$, then from (i)
\[
 h^t y = h^t n h^{-t} x \to x
\]
\end{proof}

Now we prove that

\begin{corollary}\label{teovarestavel}
The stable set of the fixed points ${\rm fix}(H, u)$ is given by
\[
{\rm st}(H, u) = N^-_H {\rm fix}(H, u)
\]
while the unstable set of the fixed points ${\rm fix}(H, u)$ is given by
\[
{\rm un}(H, u) = N^+_H {\rm fix}(H, u)
\]
and
\[
K = \bigcup_{u \in U} {\rm st}(H, u)
\]
\end{corollary}

\begin{proof}
We prove the results only for the stable set. The proof for the unstable set is entirely analogous. By Lemma \ref{lemmaestavel}, we have that $N^-_H {\rm fix}(H, u) \subset {\rm st}(H, u)$. Besides this, if $\pi$ denotes the projection from $K$ to $\mathbb{F} = K/M$ and the projection from $U$ to $W = U/C$, we have that
\[
\pi\left(N^-_H {\rm fix}(H, u)\right)
= N^-_H {\rm fix}_\F(H, w)
= {\rm st}_\F(H, w)
\]
where $w = \pi(u)$. Since ${\rm fix}_\F(H, w) = K_H^0wb_0$, we have that
\[
\pi^{-1}\left({\rm st}_\F(H, w)\right)
= N^-_HK_H|^0Mub \subset \bigcup_{u \in U} N^-_H {\rm fix}(H, u)
\]
Since
\[
\mathbb{F} = \bigcup_{w \in W} {\rm st}_\F(H, w)
\]
it follows that
\[
K = \bigcup_{w \in W} \pi^{-1}\left({\rm st}_\F(H, w)\right) \subset \bigcup_{u \in U} N^-_H {\rm fix}(H, u)
\]
which completes the proof.
\end{proof}

Now we determine the Bruhat decomposition. First note that stable manifolds are disjoint if and only the corresponding fixed points are disjoint so that to study the intersection of stable sets we only need to know the intersection of the corresponding fixed points sets.

Let us now also define \emph{the subgroup} $C_H$ of $C$.
\[
C_H: = \frac{K_H^0 \cap M}{M_0}
\]

In the general case then the set of fixed points $\{ {\rm fix}(H, u) : u \in U \}$ are in bijection with the quotient $U_H \backslash U $. Note that also generally, $U_H$ is not normal in $U$ and this quotient is not a group but only right cosets. These cosets can then be used to enumerate the components of fixed points.

\begin{theorem}\label{theorem:elemento-nao-reg}
Let $h^t$ be a hyperbolic flow in $K$.
\begin{enumerate}[$(i)$]

\item The set of fixed points is the disjoint union of connected components
\[
{\rm fix}(h^t) = \uniaodisj \left\{ {\rm fix}(H, u) : \, u \in U_H \backslash U \right\}
\]
where ${\rm fix} (H, u) = K_H^0 u b $. The attractors are given by ${\rm fix} (H, c)$ and the repellers are given by ${\rm fix}(H, c u^-)$, where ${c \in C_H \backslash C }$ and $\pi(u^-) = w^-$.

\item The group $K$ decomposes as the disjoint union of stable manifolds,

\begin{equation}\label{eq:decompos-dinamica-geral-flag}
K = \uniaodisj \left\{ {\rm st} (H, u) : \, u \in U_H \backslash U \right\}
\end{equation}

where each ${\rm st} (H, u) = N^-_H {\rm fix} (H, u) = N^-_H K_H^0 u b $. Also, the union of the stable manifolds ${\rm st} (H, c)$, where ${c \in C_H \backslash C}$, is open and dense in $K$.
\end{enumerate}
\end{theorem}

\begin{proof}
The first statement of itens $(i)$ and $(ii)$ follows from Theorem \ref{teo:ptos-fixos} and Corollaries \ref{corol:fix-flag} and \ref{teovarestavel}, noting that stable manifolds of disjoint sets are disjoint.

For the second statement of item $(ii)$, we have that ${\rm fix}_\F(H, w)$ is the unique atractor of $h^t$ on $\mathbb{F}$ and that ${\rm st}_\F(H, 1)$ is open and dense in $\mathbb{F}$. Thus, if $\pi$ denotes the projection from $K$ to $\mathbb{F} = K/M$, we have that
\[
\pi^{-1}\left({\rm st}_\F(H, 1)\right)
= N^-_HK_H^0Mb = N^-_HK_H^0Cb
\]
is open and dense in $K$. Since $N^-_HK_H^0Cb$ is a finite union of orbits $N^-_HK_H^0cb$, where $c \in C$, at least one of them is open and thus all of them are open, since they are all diffeomorphic to each other. Hence the attractors are ${\rm fix} (H, c)$ for ${c \in C_H \backslash C }$, since $C \cap U_H = C_H$. The proof for the repellers is similar, since they are the attractors of $\left(h^{-1}\right)^t$, and that $-w^-H \in {\rm cl}\mathfrak{a}^+$.
\end{proof}

Note that item (ii) from the previous Theorem illustrates that the height function on $K$ with respect to $H$ has $|M/M_0|=|C|$ components that assume the maximum and $|C|$ components that assume the minimum.

\section{General translations}

In this section, we describe the dynamics of a general translation $g^t$, given either by the iteration of some $g \in G$, for $\mathbb{T} = \Z$, or by $\exp(tX)$, for $\mathbb{T} = \R$. We show that the minimal Morse decomposition of $g^t$, as well their stable and unstable manifolds, is determined by its hyperbolic component $h^t = \exp(tH)$ and that each minimal Morse component is normally hyperbolic.

\subsection{Minimal Morse decomposition and recurrence}

We will show that the minimal Morse components of $g^t$ in $K$ are given by
\[
{\cal M}(g^t,u) := {\rm fix}(H, u) \cup g{\rm fix}(H, u)
\]
where $g = g^1 \in G_H$.

\begin{proposition}
There exists $c_g \in C$ such that $gG_H^0 = G_H^0c_g$, $g{\rm fix}(H, u) = {\rm fix}(H, c_gu)$, and
\[
{\cal M}(g^t,u) = G_H^g u b = K_H^g u b
\]
where
\[
G_H^g := G_H^0 \cup gG_H^0
\qquad \mbox{and} \qquad
K_H^g := K \cap G_H^g
\]
are subgroups of $G_H$. If $g \in G_H^0$, which always happens when $\mathbb{T} = \R$, then ${\cal M}(g^t,u)$ is connected and equal to ${\rm fix}(H, u)$. If $g \notin G_H^0$, then ${\cal M}(g^t,u)$ has two connected components ${\rm fix}(H, u)$ and $g{\rm fix}(H, u)$. In both cases, ${\cal M}(g^t,u)$ is invariant by $g^t$.
\end{proposition}

\begin{proof}
Since $g \in G_H = G_H^0M = G_H^0C$, it follow that $gG_H^0 = G_H^0c_g$, for some $c_g \in C$. Thus we have that
\[
G_H^g = G_H^0 \cup G_H^0c_g
\]
is a subgroup of $G_H$, since $C$ is a finite direct product of cyclic groups of order two. We also have that
\[
g{\rm fix}(H, u) = gK_H^0ub = gG_H^0ub = G_H^0c_gub = {\rm fix}(H, c_gu) = K_H^0c_gub
\]
Since $K_H^0 = K \cap G_H^0$, it follows that
\[
K_H^g = K_H^0 \cup K_H^0c_g
\]
and that
\[
K_H^g u b = {\cal M}(g^t,u)
\]
If $g \in G_H^0$, then $G_H^g = G_H^0$ and hence ${\cal M}(g^t,u) = {\rm fix}(H, u)$. If $g \notin G_H^0$, then $c_g \notin K_H^0$ and hence $K_H^0 \cap K_H^0c_g = \emptyset$, which shows that ${\rm fix}(H, u) \cap {\rm fix}(H, c_gu) = \emptyset$. Since $g^t \in G_H^g$, for all $t$, we have that ${\cal M}(g^t,u)$ is $g^t$-invariant.
\end{proof}

Consider the metric in $G/AN$ induced by the bi-invariant metric $d$ in $K$, such that
\[
 d(k'b, k''b) = d(k', k'')
\]
for all $k',k'' \in K$. It follows that
\[
 d(kk'mb, kk''mb) = d(k'b, k''b)
\]
for all $k \in K$ and all $m \in M$.

\begin{lemma}\label{novolemma2}
Let $G$ be a semisimple Lie group with Iwasawa decomposition $G = K A N$ with $K$ compact and $\kappa: G \to K$ so that $g \in \kappa(g) A N$. Let $u^t $ be a unipotent flow that commutes with the elliptic flow $e^t   \in K$.

\begin{enumerate}[$(i)$]
\item There is $s_n \to \infty$ so that $e^{s_n} \to 1$. 

\item For all $k \in K$, there is $a_n \in \N$ so that $k^{a_n} \to 1$.
\end{enumerate}
\end{lemma}

\begin{proof}
For item (i), since $K$ is compact there is $t_n \to \infty$ and $k \in K$ so that $e^{t_n} \to k $. Since $t_n \to \infty$ it is possible to assume taking a subsequence of $t_n$ that $s_n = t_{n+1} - t_{n} \to \infty$. By the isometry of left multiplication then 
\[
d(e^{t_{n+1}} e^{-t_{n}},1) = d(e^{t_{n+1}}, e^{t_n}) \leq d(e^{t_{n+1}}, k) + d(k, e^{t_n}) \to 0 
\]
when $n \to \infty$ and then $e^{s_n} \to 1 $ when $n \to \infty$. For item (ii), in a similar form to the previous item, taking the sequence $k^n \in K$ where $n \in \N$, there is a sequence $a_n \in \N$ so that $k^{a_n} \to 1$.
\end{proof}

\begin{lemma}\label{novolemma3}
Let $F$ be a compact subset of $K \simeq G/AN$ invariant by $u^t$. For all $\epsilon >0$, $kb \in F$ and $T>0$, there exist $s > T$, $k_0 \in F$ and $m \in M$, so that, $e^s \in B(1, \epsilon)$, $ u^{-s}kb \in B(k_0b, \epsilon)$ and $ u^s kb \in B(k_0mb, \epsilon)$.
\end{lemma}

\begin{proof}
From item (i) of Lemma \ref{novolemma2} there is $n_0 \in \N$ so that $e^{s_n} \in B (1, \epsilon)$ for $n > n_0$. For all $kb \in F$ there is $k'_0 \in K$ so that $u^t k b_0 \to k'_0 b_0$ when $ t \to \pm \infty$. Since $K$ and $\pi^{-1}(k'_0b_0)$ are compact, we can take a subsequence of $s_n$ and assume that $u^{-s_n} kb \to k'_0 m'b \in \pi^{-1}(k'_0b_0)$, for some $m' \in M$. We can take a subsubsequence of $s_n$  and assume that  $u^{s_n} kb \to k'_0 m''b \in \pi^{-1}(k'_0b_0)$, for some $m'' \in M$.
Then there exists an integer $N > T$ such that $e^{s_N} \in B (1, \epsilon)$, $u^{-s_N} kb \in B (k'_0 m'b, \epsilon)$ and $ u^{s_N} kb \in B (k'_0 m''b, \epsilon)$, now take $s = s_N$, $k_0 = k'_0 m'$ and $m = (m')^{-1} m''$ then $k'_0 m'' = k'_0 m' (m')^{-1} m'' = k_0 m$.
\end{proof}

\begin{proposition}\label{recorrente por cadeia completo}
Let $F$ be a compact subset of $K \simeq G/AN$ invariant to the left by $e^t$ and $u^t$ and invariant to the right by $M$. Then $\mathcal{R}_C (e^t u^t|_F) = F$.
\end{proposition}

\begin{proof}
Let $kb \in F$ and $\epsilon > 0$, to prove the result we will show a $(\epsilon,T)$-chain from $k$ to $k$. First, let us form a $(\epsilon/2,T)$-chain from $k$ to $k m^p$.
Take $\epsilon$ in Lemma \ref{novolemma3} to be $\epsilon/8$. Then there are $s >T$, $k_0b \in F$ and $m \in M$, so that, $e^s \in B (1, \epsilon/8)$,
\[
u^{-s} kb \in B (k_0b,\epsilon/8) \quad \mbox{ and } \quad u^{s} kb \in B (k_0 mb,\epsilon/8)
\]
From invariance, we have that
\[
 d(e^{-s}u^{-s}kb, e^{-s}k_0b) = d(u^{-s}kb , k_0b)
  \quad \mbox{ and } \quad
 d(e^{s}u^{s}kb, e^{s}k_0b) = d(u^{s}kb , k_0b)
\]
and that
\[
 d(e^{-s}k_0b, k_0b) = d(e^{-s}k_0, k_0) = d(e^{-s}, 1) = d(1, e^s)
\]
and
\[
 d(e^{s}k_0b, k_0b) = d(e^{s}k_0, k_0) = d(e^{s}, 1)
\]
From the triangle inequality,
\[
e^{-s} u^{-s} kb \in B (k_0b,\epsilon/4) \quad \mbox{ and } \quad e^{s} u^{s} kb \in B(k_0 mb,\epsilon/4)
\]
For $ i = 1, \ldots, p-1$, we have that
\[
e^{-s} u^{-s} k m^ib \in B (k_0 m^ib , \epsilon/4) \quad \mbox{ and } \quad e^{s} u^{s} k m^ib \in B (k_0 m^{i+1}b,\epsilon/4)
\]
Now we have that the sequence
\[
\{ kb, e^{-s} u^{-s} k mb, kmb, e^{-s} u^{-s} k m^2b, km^2b, \ldots, k m^pb \}
\]
is a $(\epsilon/2, T)$-chain from $kb$ to $k m^pb$. In fact, note that
\begin{eqnarray*}
d(e^{s}u^{s}kb, e^{-s}u^{-s}kmb) &\leq& d(e^{s}u^{s}kb, k_0mb) + d(k_0mb , e^{-s} u^{-s}kmb) \\
&< & \frac{\epsilon}{4} + \frac{\epsilon}{4} = \frac{\epsilon}{2}
\end{eqnarray*}
Then it follows that
\[
d(e^{s}u^{s}km^ib, e^{-s}u^{-s}km^{i+1}b) < \frac{\epsilon}{2}
\]
This proves that the previous sequence is indeed a $(\epsilon/2, T)$-chain from $kb$ to $k m^pb$. Now take $p$ to be an positive integer so that $m^p \in B (1, \epsilon/2)$. Thus $m^pb \in B (b, \epsilon/2)$ so that $k m^pb \in B(kb, \epsilon/2)$. Replacing the last term $k m^pb$ in the previous chain by $kb$ we obtain an $(\epsilon,T)$-chain from $kb$ to $kb$.
\end{proof}

\begin{theorem}\label{propmorseflag}
 Then
\begin{enumerate}[$(i)$]
\item The minimal Morse components of $g^t$ are given by ${\cal M}(g^t,u)$ and
\[
{\cal R}_C (g^t) = {\rm fix} (h^t) = \uniaodisj \left\{ {\cal M}(g^t,u): \, u \in U_H^g \backslash U \right\}
\]
where
\[
U_H^g := U_H \cup c_gU_H
\]
The attractors are given by ${\cal M}(g^t,c)$ and the repellers are given by ${\cal M}(g^t,c u^-)$, for $c \in C_H^g \backslash C$, where
where
\[
C_H^g := C_H \cup c_gC_H
\]

\item The stable manifolds of $g^t$ are given by ${\rm st}(g^t,u) := N^-_H {\cal M}(g^t,u)$ and the unstable manifolds of $g^t$ are ${\rm un}(g^t, u) := N^+_H {\cal M}(g^t,u)$. The union of the stable manifolds of the attractors is open and dense.
\end{enumerate}
\end{theorem}

\begin{proof}
First we will prove that ${\rm st}(g^t,u)$ is a stable set from the $g^t$-invariant ${\cal M}$. By the Bruhat decomposition of $K$ (Theorem \ref{theorem:elemento-nao-reg} (ii)), it is enough to show that ${\rm st}(g^t,u)$ is in the stable manifold of ${\cal M}(g^t,u)$, that is, it is enough to prove that for any $y \in S$ then $\omega(y) \subset {\cal M}(g^t,u)$. Let $y \in {\rm st}(g^t,u)$, so $y = \exp(Y) l u b$, where $Y \in \n^-_H$ and $l \in G_H^g$. Then
\[
g^t y = g^t \exp(Y) g^{-t} g^t l u b = \exp(g^t Y) g^t l u b
\]
where $g^t l u b \in {\cal M}(g^t,u)$, since $g^t l \in G_H^g$. First, we prove that $g^t Y \to 0$. In fact, from the Jordan decomposition, $r (g)$ is the greatest eigenvalue from its hyperbolic component, is given by the restriction of $h$ to $\n^-_H$. These eigenvalues are given by $e^{-\alpha (H)}$, where $\alpha \in \Pi^+$ and $\alpha(H) > 0$, hence $r (g) < 1$. Taking then $x \in \omega(y)$ then let $t_j \to \infty$ such that $g^{t_j} y \to x$. Since $\exp(g^{t_j} Y) \to 1$, then
\[
g^{t_j} y = \exp(g^{t_j} Y) g^{t_j} l u b \to x
\]
and $x$ is in ${\cal M}(g^t,u)$, since $g^{t_j} l u b$ is in ${\cal M}(g^t,u)$, which is compact. Hence $\omega(y) \subset {\cal M}(g^t,u)$. Similarly, we can prove that ${\rm un}(g^t, u)$ is the unstable set of ${\cal M}(g^t,u)$.

Now we have that $\{{\cal M}(g^t,u): u \in U \}$ is a Morse decomposition for $h^t$ in $K$, since we have that ${\cal M}(g^t,u) = {\rm fix}(H, u) \cup {\rm fix}(H, c_gu)$, that ${\rm st}(g^t,u) = {\rm st}(H, u) \cup {\rm st}(H, c_gu)$, that ${\rm un}(g^t,u) = {\rm un}(H, u) \cup {\rm un}(H, c_gu)$, and that there is no $x$ such that $\omega(x) \cap {\rm fix}(H, cu) \neq \emptyset$ and $\omega^*(x) \cap {\rm fix}(H, c'u) \neq \emptyset$, for any $c \neq c' \in C$. In fact, we have that
\[
 \pi^{-1}(\pi({\rm fix}(H, u))) = \bigcup_{c \in C} {\rm fix}(H, cu)
\]
so that, if $x \notin {\rm fix}(H, cu)$, for any $c \in C$, then $\pi(x) \notin \pi({\rm fix}(H, u)) = {\rm fix}_\F(H, \pi(u))$, so that $\omega(\pi(x)) \cap {\rm fix}_\F(H, \pi(u)) = \emptyset$ or $\omega^*(\pi(x)) \cap {\rm fix}_\F(H, \pi(u)) = \emptyset$, since $\{{\rm fix}_\F(H, \pi(u)): u \in U \}$ is a Morse decomposition for $h^t$ in $\mathbb{F}$. This implies that $\{{\cal M}(g^t,u): u \in U \}$ is a Morse decomposition for $g^t$ in $K$, since the stable and unstable sets of each ${\cal M}(g^t,u)$ with respect to $g^t$ or to $h^t$ coincide. Hence
\[
{\cal R}_C (g^t) \subset {\rm fix}(h^t) = \bigcup_{u \in U} {\cal M}(g^t,u)
\]
By Proposition \ref{recorrente por cadeia completo}, we have that $ {\rm fix}(h^t) = \mathcal{R}_C (e^tu^t|_{{\rm fix}(h^t)}) = \mathcal{R}_C (g^t|_{{\rm fix}(h^t)}) \subset \mathcal{R}_C (g^t)$ and thus $\mathcal{R}_C (g^t) = {\rm fix} \left( h^t \right)$. In order to show that ${\cal M}(g^t,u)$ are the minimal Morse components of $g^t$ in $K$, we need just to remember that each connected component of $\mathcal{R}_C (g^t)$ is transitive and, when ${\cal M}(g^t,u)$ is disconnected, $g = g^1$ switches its two connected components ${\rm fix}(H, u)$ and $g{\rm fix}(H, u)$. Since ${\cal M}(g^t,u) = {\rm fix}(H, u) \cup {\rm fix}(H, c_gu)$, the correct quotient to get a disjoint union is $U_H^g \backslash U$ instead of $U_H \backslash U$. The other statements of items (i) and (ii) follow from Theorem \ref{theorem:elemento-nao-reg}, since the stable and unstable sets of each ${\cal M}(g^t,u)$ with respect to $g^t$ or to $h^t$ coincide.
\end{proof}

Next we obtain a characterization of the recurrent set.

\begin{theorem}\label{teorecflag}
Let $g^t$ be a translation flow on $K \simeq G/AN$ and $g^t = e^t h^t u^t$ be its Jordan decomposition. The recurrent set of translations $g^t$ induced in $K$ is given by
\[
\mathcal{R}(g^t) = {\rm fix} (h^t) \cap {\rm fix} (u^t)
\]
\end{theorem}

\begin{proof}
First we prove that $ {\rm fix} ( h^t ) \cap \mathcal{R} ( u^t ) \subset \mathcal{R}(g^t )$. If $kb \in \mathcal{R} ( u^t )$ then $kb_0$ is a fixed point of the maximal flag $ \F = K/M$, so $u^t kb = k m_tb$ with $m_t \in M$ for all $t \in \R$. Now we have that $m_{t+s} = m_t m_s$. Indeed, we have that
\[
 k m_{t+s}b = u^{t+s} kb = u^t u^s kb
  = u^t k m_sb = u^t kb m_s = k m_tb m_s
  = k m_tm_sb
\]
showing that $ m_{t+s} = m_tm_s $ for all $t,s \in \R$. Since $m_0 = 1$ then $m_{-t} = m_t^{-1}$. And since $M$ is compact, there exist $t_i \to \infty$ and $\widetilde{m} \in M$ so that  $u^{t_i} kb \to k \widetilde{m}b$. Taking a subsequence of $t_i$ we can assume that $e^{t_i} \to \widetilde{e}$ and $s_i = t_{i+1} - t_i \to \infty$. Then
\[
m_{s_i} = m_{t_{i+1}} m_{- t_i} \to \widetilde{m} \widetilde{m}^{-1} =1
\]
and $u^{s_i} kb = k m_{s_i}b \to kb$. Since $e^{s_i} \to \widetilde{e}{\widetilde{e}}^{-1} = 1$ then
\[
g^{s_i} kb = e^{s_i} u^{s_i} kb \to kb
\]
showing that $kb \in \mathcal{R} (g^t)$.
Now let us prove that $\mathcal{R} (g^t) \subset {\rm fix} (h^t) \cap \mathcal{R} (u^t)$. From item (i) of Theorem \ref{propmorseflag}, $ kb \in \mathcal{R} (g^t)\subset \mathcal{R}_C (g^t) = {\rm fix}\left( h^t \right)$, so we just need to prove that $kb \in \mathcal{R} (u^t)$. Let $t_i \to \infty$ such that $g^{t_i} kb \to kb $ and taking a subsequence we can assume that $e^{t_i} \to \widetilde{e}$ and that $s_i = t_{i+1} - t_i \to \infty$. Thus
\[
 u^{s_i} kb = e^{-s_i}g^{s_i} kb \to kb
\]
since $e^{-s_i} \to 1$.

Since ${\rm fix} (u^t) \subset \mathcal{R} (u^t)$, we now show that $\mathcal{R} (u^t) \subset \pi^{-1}({\rm fix}_\F (u^t)) \subset {\rm fix} (u^t)$. For the first inclusion, if $u^{t_n} kb \to kb$, when $t_n \to \infty$, then $u^{t_n}kb_0 \to kb_0 \in \F$ so that $kb_0 \in {\rm fix}_\F (u^t)$. For the second inclusion, if $kb$ is such that $kb_0$ is a fixed point in $\F$, then $u^t kb = k m_tb$ with $m_t \in M$. Equivalently, we have that $\kappa(k^{-1}u^t k) = m_t$. Now, if we aplly to this equation the homomorphism $g \mapsto \widehat{g}$ given by the adjoint map from $G$ to the adjoint group $\widehat{G}$, we get $\widehat{\kappa}\left(\widehat{k}^{-1}\widehat{u^t} \widehat{k}\right) = \widehat{m_t}$, where the projection $\widehat{\kappa}:\widehat{G} \to \widehat{K}$ is given by the Gram-Schmidt process applied to the columns of the matrix relative to a suitable basis (Lemma 6.45 of \cite{knapp}). We have that the entries of the matrix $\widehat{u^t}$ are polynomials in $t$, since $\widehat{u^t}$ is the exponential of a nilpotent matrix times $t$. Hence the square of the entries of the matrix $\widehat{\kappa}\left(\widehat{k}^{-1}\widehat{u^t} \widehat{k}\right)$ are rational fuctions in $t$, where the degree of the denominator is greater than or equal to the degree of the numerator. Thus $\widehat{\kappa}\left(\widehat{k}^{-1}\widehat{u^t} \widehat{k}\right) = \widehat{m_t}$ has a well defined limit when $t \to \infty$, which we denote by $\widehat{m}$. Since $\widehat{m_t}\widehat{m_s} = \widehat{m_{t+s}}$, taking limit when $s \to \infty$, we get that $\widehat{m_t}\widehat{m} = \widehat{m}$ and thus $\widehat{m_t} = 1$. Hence we have that $m_t$ is in the kernel of the adjoint map, which is discrete, since it is the center of $G$. Since $m_0 = 1$ and $t \in \R$, it follows that $m_t = 1$ and that $kb \in {\rm fix} (u^t)$, completing the proof.
\end{proof}

\subsection{Normal hyperbolicity}\label{seclineariz}

In this section, we follow the same steps used in \cite{Patrao1} in order to show that each ${\cal M} = {\cal M}(g^t,u) = G_H^g u b = K_H^g u b$ is normally hyperbolic. First we construct an appropriate Riemannian metric of $K$. For this, we define an adequate complement of the isotropy subalgebra $\g_x$ for $x$ in $K$ to use it as a model for the tangent bundle $T K_x$. Let us fix in $\g$ a Cartan inner product $\langle \cdot, \cdot \rangle$, that is $K$-invariant. Consider $\g_x^\perp$ the orthogonal complement of $\g_x$ in $\g$ with respect to this inner product. The map
\begin{equation}
\label{normalhypeq1}
\g_x^\perp \to T K_x \qquad X \mapsto X \cdot x
\end{equation}
is a linear isomorphism, and, for $k \in K$, we have that
\begin{equation}
\label{eqlinearizacao1}
k ( \g_x^\perp ) = \g_{k x}^\perp
\end{equation}
Now, since $\g_{k b} = k ( \a \oplus \n )$ and $(\a \oplus \n)^\perp = \m \oplus \n^- $, it follows that
\begin{equation}
\label{eqperpisotropia2}
\g_{k b}^\perp = k (\m \oplus \n^-)
\end{equation}
The following result is proved in the same way as Proposition 3.1 of \cite{Patrao1}

\begin{proposition}
\label{normanatural}
The following
\[
\langle X \cdot x ,\, Y \cdot x \rangle_x := \langle X, Y \rangle \qquad\text{for } X , Y \in \g_x^\perp
\]
defines a $K$-invariant Riemannian metric in $K$ such that the map defined in {\normalfont (\ref{normalhypeq1})} is an isometry. Then, for $Y \in \g$,
\[
| Y \cdot x |_x \leq |Y|
\]
with equality if, and only if, $Y \in \g_x^\perp$.
\end{proposition}

From now on, let us fix the previous metric in $K$. Note that, in general, this metric is different from the Borel metric used previously. First we define candidates for the stable and unstable bundles orthogonal to the tangent bundle of ${\cal M}$. The tangent space of the orbit
\[
{\cal M} = G_H^g u b
\]
is given by
\[
T {\cal M} = G_H^g (\g_H \cdot u b) \subset T K
\]
Consider the orthogonal decomposition
\begin{equation} \label{eqdecomposicaog_H}
\g =\g_H \oplus \n^-_H \oplus \n^+_H
\end{equation}
where $\g_H$ and $\n^\pm_H$ are $G_H$-invariant. {\em Define}
\[
V^\pm := G_H^g ( \n_H^\pm \cdot u b) \subset T K
\]
for $x \in {\cal M}$, and subspaces
\[
\v^\pm_x := \n^\pm_H \cap \g_x^\perp
\]

\begin{proposition}
\label{fibradosnormais}
\begin{enumerate}[(i)]
\item The tangent space of $K$ in ${\cal M}$ can be decomposed as the Whitney orthogonal sum,
\[
T K|_{\cal M} = T {\cal M} \oplus V^- \oplus V^+
\]
where $V^\pm$ are $G_H^g$-invariant differentiable vector sub-bundles of ${\cal M}$. And, in particular, $V^- \oplus V^+$ is a normal fiber of $T {\cal M}$.

\item For $x \in {\cal M}$, the map
\[
\v^\pm_x \to V^\pm_x \qquad  Y \mapsto Y \cdot x
\]
is a linear isomorphism and $k (\v_x) = \v_{k x}$ for $k \in K_H$.
\end{enumerate}
\end{proposition}

\begin{proof}
The $G_H^g$-invariance is immediate from the definitions. The proof that $V^\pm$ aresub-bundles over ${\cal M} = G_H^g u b$ is similar to the proof of Proposition 3.2 of \cite{Patrao1}. From the  $G_H^g$-invariance of $\n^\pm_H$ it follows that
\begin{equation}
\label{fibradosnormaiseq0}
V^\pm_x = \{ (g \n_H^\pm \cdot g ub): \, g \in G_H^0, g ub = x \} = \n_H^\pm \cdot x
\end{equation}
To prove the Whitney sum, by the orthogonal decomposition
\[
\g = u (\m \oplus \n^-) \oplus u (\a \oplus \n)
\]
by the decomposition in root spaces of $\n^\pm_H$ and by equation (\ref{eqperpisotropia2}) then
\begin{eqnarray*}
\n^\pm_H & = & \left(\n^\pm_H \cap u (\m \oplus \n^-) \right) \oplus \left(\n^\pm_H \cap u (\a \oplus \n) \right) \nonumber\\
& = & (\n^\pm_H \cap \g^\perp_{u b}) \oplus (\n^\pm_H \cap \g_{u b})
\end{eqnarray*}
For $x \in {\cal M}$, then $x = k u b$ with $k \in K_H^g$. From equation (\ref{eqlinearizacao1}) then $k (\g_{u b}^\perp) = \g_x^\perp$ and $k  \g_{u b} = \g_x$. Since $K_H^g = K \cap G_H^g$ normalizes $\n^\pm_H$, taking $k$ in both sides of the previous decomposition we get
\begin{align}
\n^\pm_H & = (\n^\pm_H \cap \g_x^\perp) \oplus (\n^\pm_H \cap \g_x) \nonumber \\
& = \v^\pm_x \oplus (\n^\pm_H \cap \g_x)
\label{fibradosnormaiseq2}
\end{align}
From this and from equation (\ref{fibradosnormaiseq0}) then
\[
\v^\pm_x \to V^\pm_x \qquad Y \mapsto Y \cdot x
\]
is a linear isomorphism. Since $\g_H$ is also $G_H^g$-invariant, the same argument applies to $T{\cal M}$ to get $T{\cal M}_x = \g_H \cdot x$ and
\begin{equation}
\label{fibradosnormaiseq3}
\g_H = (\g_H \cap \g_x^\perp) \oplus (\g_H \cap \g_x)
\end{equation}
and
\begin{equation}
\label{fibradosnormaiseq4}
T{\cal M}_x = (\g_H \cap \g_x^\perp) \cdot x
\end{equation}
Combining the decompositions (\ref{eqdecomposicaog_H}), (\ref{fibradosnormaiseq2}), and (\ref{fibradosnormaiseq3}), we get the orthogonal sum
\[
\g_x^\perp = (\g_H \cap \g_x^\perp) \oplus \v_x^- \oplus \v_x^+
\]
The image of this decomposition by the isometry (\ref{normalhypeq1}) is the orthogonal sum
\[
T K_x = T{\cal M}_x \oplus V^-_x \oplus V^+_x
\]
Since $k \in K_H^g \subset K_H$ normalizes $\n^\pm_H$, from equation (\ref{eqlinearizacao1}) then $k \v_x = \v_{k x}$.
\end{proof}

The next result shows that each minimal Morse component ${\cal M} = {\rm fix}(H, u) $ of $g^t$ in $K$ is normally hyperbolic.

\begin{proposition} \label{propnormanatural-g}
The vector bundles $V^\pm$ are $g^t$-invariant and there are positive numbers $c$ and $\nu < \lambda$ such that
\begin{enumerate}[(i)]
\item  $|g^t v| \leq c\, {\rm e}^{-\lambda t} |v|$ for $v \in V^-$and $t \geq 0$.

\item  $|g^{-t} v| \leq c\, {\rm e}^{-\lambda t} |v|$ for $v \in V^+$and $t\geq 0$.

\item  $|g^t v| \leq c\, {\rm e}^{\nu |t|} |v|$ for $v \in T{\cal M}$ and $t \in \R$.
\end{enumerate}
\end{proposition}
\begin{proof}
First note that 
\[
g^t = e^t h^t u^t
\]
where $g^t, e^t, h^t, u^t \in G_H^g$. We have that $V^\pm = \n_H^\pm \cdot {\cal M}$, that $G_H^g$ normalizes $\n_H^\pm$, and $G_H^g$ also leaves ${\cal M}$ invariant, so that $V^\pm$ is $g^t$-invariant. Let $v \in V^\pm$, then $v = Y \cdot x$ with $Y \in \v^\pm_x = n_H^\pm \cap \g_x^\perp$ and $x \in {\cal M}$, from Proposition \ref{normanatural}, $|v| = |Y|$ and
\[
|g^t v| = |g^t Y \cdot g^t x| \leq |g^t Y|
\]
where $g^t Y \in \n_H^\pm$. It is enough to show then the inequalities for $g^t$ restricted to $\n_H^\pm$. First we consider the case where $Y \in \n_H^-$, the next case is proven similarly. From Lemma \ref{lemmadecaimentoexp}, there is $\mu >0$ such that $|h^t Z| \leq \e^{-\mu t} |Z|$, for $t \geq 0$ and $Z \in \n_H^-$. Since $e^t \in K_H^0$ and the inner product is $K$-invariant, then
\[
|g^t Y| = |h^t u^t Y| \leq \e^{-\mu t} |u^t Y|
\]
where we used that $u^t \in G_H^g$, so $u^t Y \in \n^-_H$. Since $u^t = \exp(t N)$, for $N \in \g$ nilpotent, it follows that $u^t Y = \e^{t\ad(N)} Y$. From the triangle inequality then
\[
|u^t Y| = |\e^{t\ad(N)} Y| \leq \sum_{k \geq 0} \frac{|t^k|}{k!} \| \ad(N)^k \| |Y| = p(t) |Y|
\]
where $\| \cdot \|$ is the operator norm associated to the norm $|\cdot |$ in $\n_H^-$ and $p(t)$ is a polynomial, since $\ad(N)$ is nilpotent. Then
\[
|g^t Y| \leq \e^{-\mu t} p(t) |Y|
\]
Since $|g^t v| \leq |g^t Y|$ and $|v| = |Y|$ then for $v \in V^-$
\[
|g^t v| \leq \e^{-\mu t} p(t) |v|,\quad t \geq 0
\]
The case for $V^+$ is similar so that for $v \in V^+$
\[
|g^{-t} v| \leq \e^{-\mu t} p(t) |v|,\quad t \geq 0
\]
For $T {\cal M}$, note that $x \in {\cal M}$ and $g^t x = e^t u^t x$, and $g^t$ acts as $e^t u^t$ in $T {\cal M}$. From equation \ref{fibradosnormaiseq4} from Proposition \ref{fibradosnormais} a tangent vector $v \in T{\cal M}_x$ is $v = Y \cdot x$, for $Y \in \g_H \cap \g_x^\perp$. From Proposition \ref{normanatural} $|v| = |Y|$ and 
\[
|g^t v| = |e^t u^t Y \cdot e^t u^t x| \leq |e^t u^t Y| = |u^t Y| \leq p(t) |Y| = p(t) |v|
\]
where we used that $e^t \in K_H^g$ and the inequality for $|u^t Y|$ previously obtained.

Since $\e^{-\frac{\mu}{2} |t|} p(t) \to 0$ when $t \to \pm \infty$, then $\e^{-\frac{\mu}{2} |t|} p(t)$ is bounded by some $c_1 > 0$ and
\[
\e^{-\mu t} p(t) = \e^{-\frac{\mu}{2} t} \left( \e^{-\frac{\mu}{2} |t|} p(t) \right) \leq c_1  \e^{-\frac{\mu}{2} t}, \quad t \geq 0
\]
For the last case, since $\e^{ -\frac{\mu}{4} |t|} p(t) \to 0$ when $t \to \pm\infty$, then $\e^{ -\frac{\mu}{4} |t|} p(t)$ is is bounded by some $c_2 > 0$  and
\[
p(t) = \e^{\frac{\mu}{4} |t|} \left( \e^{ -\frac{\mu}{4} |t|} p(t) \right) \leq c_2 \e^{\frac{\mu}{4} |t|}, \quad t \in \R
\]
The items (i), (ii) and (iii) of the Proposition follow taking $\lambda = \frac{\mu}{2}$, $\nu = \frac{\mu}{4}$, and $c$ to be the maximum of $c_1$ and $c_2$.
\end{proof}

It follows that $V^\pm$ are the unstable/stable bundle of $g^t$.
By the main result of \cite{pugh}, we obtain a linearization of this flow in a neighborhood of each minimal Morse component ${\cal M}(g^t,u)$.

\begin{corollary}\label{corlinearizgt}
Let $V = V^- \oplus V^+$.
There exists a differentiable map $ V \to
K$ which takes the null section to
${\cal M}(g^t,u)$ and such that:

\begin{enumerate}[$(i)$]
\item Its restriction to some neighborhood of the null section $V_0$ inside $V$ is a $g^t$-equivariant diffeomorphism onto some neighborhood of ${\cal M}(g^t,u)$ inside $K$.

\item Its restrictions to $V^\pm$ are $g^t$-equivariant diffeomorphisms, respectively, onto the unstable/stable manifolds
$N^\pm_H {\cal M}(g^t,u)$.
\end{enumerate}
\end{corollary}
\begin{proof}
It is enough to note that the action of $g^t$ on $V$ is given by the restriction of the differential of the action of $g^t$ on $K$ and also that the equivariance property is equivalent to the conjugation property of \cite{pugh}.
\end{proof}

\end{document}